\documentclass[12pt,reqno]{amsart}

\usepackage{amsmath,amsfonts,amsthm,mathrsfs,amssymb,amscd,enumerate,url,tikz-cd}

\numberwithin{equation}{section}

\input{xypic}
\xyoption{all}

\newtheorem{theorem}{Theorem}[section]

\newtheorem{corollary}[theorem]{Corollary}
\newtheorem{claim}[theorem]{Claim}

\newcommand{\Z}{\mathbb{Z}}
\newcommand{\C}{\mathbb{C}}

\newcommand{\Q}{\mathbb{Q}}

\newcommand{\N}{\mathbb N}

\newcommand{\E}{\mathcal{E}}
\newcommand{\F}{\mathcal{F}}

\newcommand{\Spec}{\operatorname{Spec}}

\newcommand{\mf}[1]{\mathfrak{#1}}

\newcommand{\mc}[1]{\mathcal{#1}}
\renewcommand{\t}[1]{\widetilde{#1}}

\newcommand{\Y}{\mathscr{Y}}
\newcommand{\X}{\mathscr{X}}

\usepackage{marginnote}
\usetikzlibrary{calc}

\begin{document}
\title[Direct image and pullback of Parabolic bundles]{Direct image and pullback of Parabolic vector bundles}

\author[I. Biswas]{Indranil Biswas} 
	
\address{Department of Mathematics, Shiv Nadar University, NH91, Tehsil Dadri, Greater Noida, Uttar
Pradesh 201314, India} 
	
\email{indranil.biswas@snu.edu.in, indranil29@gmail.com} 
	
\author[C. Gangopadhyay]{Chandranandan Gangopadhyay} 

\address{Department of Mathematics, Shiv Nadar University, NH91, Tehsil Dadri, Greater Noida, Uttar
Pradesh 201314, India} 

\email{chandranandan.g@snu.edu.in}

\begin{abstract}
Niels Borne established a natural correspondence between the parabolic vector bundles on curves
and vector bundles on root stacks. The notions of direct image of parabolic vector bundles and
pullback of parabolic vector bundles were studied in \cite{Alfaya_Biswas}. We show that
these two notions correspond to the notions direct image of vector bundles on root stacks and
pullback of vector bundles on root stacks respectively. Some applications of this correspondence
are given.
\end{abstract}



\maketitle

\section{Introduction}

The notion of parabolic vector bundles on curves was introduced by Mehta and Seshadri in \cite{MS}. Over time, it
has turned out to be very useful in numerous contexts. Cadman studied line bundles on root stacks
\cite[Theorem 4.1]{Ca}. Borne gave a natural correspondence between parabolic vector bundles on curves and vector bundles
on root stacks \cite{Borne}. The notions of direct image of parabolic vector bundles and
pullback of parabolic vector bundles were studied in \cite{Alfaya_Biswas}.

Our aim here is to show that
the above mentioned correspondence of Borne takes direct images of parabolic vector bundles
(respectively, pullbacks of parabolic vector bundles) to the usual direct images (respectively, pullbacks)
of vector bundles on root stacks. See Theorem \ref{main theorem} for precise statements.

Theorem \ref{main theorem} has some consequences for parabolic orthogonal and parabolic symplectic bundles
(see Corollary \ref{cor1} and Corollary \ref{cor2}).

\section{Parabolic vector bundles}

We recall the definition of parabolic vector bundles following \cite[Definition 1]{Borne}, \cite{MS}.
Take a smooth projective curve $X$ over $\C$. Fix ordered $m$ distinct points $D_X\,=\,(x_1,\,x_2,\,\cdots,\,x_m)$
of $X$, so $x_i\,\in\, X$ for all $1\, \leq\, i\, \leq \, m$ and $x_i\,\neq\, x_k$ for all $i\,\neq\, k$. Take 
${\bf r}\,=\, (r_1,\,r_2,\,\cdots,\, r_m)\,\in\, \N^m\,.$ A parabolic vector
bundle $E_*$ on $(X,\,D_X)$ with weights in $\frac{1}{\bf r}\Z^m\,:=\,
\prod\limits_i \frac{\Z}{r_i}$ consists of a vector bundle $E$ on $X$ equipped with decreasing filtrations
\begin{equation}\label{eqn_filtration}
E_{X,x_i}\,=\,E^0_{i}\,\supset\, E^{1}_i\,\supset\, \cdots \,\supset\, E^{r_i}_i\,=\,\varpi_{x_i}E_{X,x_i}
\end{equation}
for each $1\, \leq\, i\, \leq \, m$. Here $\varpi_x$ is a generator of the maximal ideal $\mf m_x$ of $\mc O_{X,x}$. The
vector bundle $E$ is called the underlying vector bundle of $E_*$. Define the parabolic weight of $E_i^j$
to be $\frac{\alpha^j_i}{r_i}$, where $\alpha^j_i\,:=\, \max\{k\geq j\,\,\big\vert\,\, E_i^k\,=\,E_i^j\}$. Define 
$$S\ :=\ \{\frac{\alpha^i_j}{r_i}\,\,\big\vert\, \, 0\,\leq\, j\,<\, r_i,\,\, \alpha^j_i\,\neq\, r_i\}
\,=\,\{\alpha_1\,<\,\alpha_2\,<\,\cdots\,<\, \alpha_k\}\,\subset\, \frac{1}{r_i}\Z\cap [0,\, 1).$$
Note that the data of the filtration (\ref{eqn_filtration}) is same as the data of the set of
weights $S$ together with decreasing the filtration
$$E_{X,x_i}\,=\,E^{r_i\alpha_1}_{i}\,\supsetneq\, E^{r_i\alpha_2}_i\,\supsetneq\, \cdots
\,\supsetneq \,E^{r_i\alpha_k}_i\,\supsetneq\, \varpi_{x_i}E_{X,x_i}.$$
Similarly, we have the notion of morphisms between two parabolic vector bundles: given two parabolic vector bundles
$E_*$ and $F_*$ on $(X,\,D_X)$ with weights in $\frac{1}{\bf r}\Z^m$, a morphism from $E_*\,\longrightarrow\, F_*$
from $E_*$ to $F_*$ is a morphism of the underlying vector bundles 
$$\alpha\,:\, E\,\longrightarrow\, F$$ 
that preserves the filtrations (see \eqref{eqn_filtration}), in other words, $\alpha(E^j_i)
\,\subset\, F^j_i$ for all $i,\,j$.
Denote by ${\rm Par}_{\frac{1}{\bf r}}(X,\,D_X)$ the category of Parabolic vector bundles on $(X,\,D_X)$
with weights in $\frac{1}{\bf r}\Z^m$.

\subsection{Parabolic vector bundles and vector bundles on root stacks}\label{section_Parabolic bundles and bundles on root stacks}

Take $X$, $D_X$ and ${\bf r}$ as before.
Let $$\X\ :=\ \sqrt[{\bf r}]{D_X/X}$$ be the root stack associated to the tuple $({\bf r},\, D_X)$
\cite[Definition 2.1]{Borne_Laaroussi}. Denote by ${\rm Vect}(\X)$ the category of vector bundles on $\X$.
Then by \cite[Theorem 2.4.7]{Borne} we have an equivalence of categories 
$${\rm Vect}(\X)\ \xrightarrow{\,\,\,\sim\,\,\,}\ {\rm Par}_{\frac{1}{\bf r}}(X,D),\ \ \,
\E\, \longmapsto\, \widehat{\E}.$$ The construction of this equivalence in one direction is as follows.
Take a vector bundle $\E$ over $\X$. Then we have a fibered diagram
\[
\begin{tikzcd}
\X_i\ :=\ \left[{\Spec} \dfrac{\mc O_{X,x_i}[T]}{(T^{r_i}-\varpi_{x_i})}\middle/\mu_{r_i} \right] \ar[r] \ar[d] & \X \ar[d,"\pi_X"] \\
\Spec \mc O_{X,x_i} \ar[r] & X 
\end{tikzcd}
\]
Here $\mu_{r_i}$ is the group of $r_i$--th roots of unity.
In particular, restricting $\mc E$ to $\X_i$ we get a $\mu_{r_i}$--equivariant sheaf on
${\Spec}\, \frac{\mc O_{X,x_i}[T]}{(T^{r_i}-\varpi_{x_i})}$. Now, any $\mu_{r_i}$--equivariant module $M$ on
$\frac{\mc O_{X,x_i}[T]}{(T^{r_i}-\varpi_x)}$ is $\mu_{r_i}$--graded with $M\,=\,
\bigoplus\limits_{j=0}^{r_i-1}M_i$, where $M_j\,=\,\{m\,\in\, M\,\,\big\vert\,\, g\cdot m\,=\,
\exp({\frac{2\pi \sqrt{-1} j}{r_i}})m\}$ and the action of $T$ induces a graded homomorphism $M\,
\longrightarrow\, M[1]$, in other words, there are inclusion maps $$M_0 \,\hookrightarrow\, M_1
\,\hookrightarrow\, M_2 \,\hookrightarrow\, \cdots \,\hookrightarrow\, M_{r_i} \,\hookrightarrow\, M_0$$
such that the entire composition coincides with multiplication by $T^{r_i}\,=\,\varpi_{x_i}$. Therefore,
defining $M^i\,:=\,M_{n-i}\,\subset\, M_0$, we get a parabolic filtration
$$M^0\,\supset \,M^{1}\,\supset\, M^{2}\,\supset\, \cdots \,\supset\, \varpi_{x_i}M^{0}.$$

\subsection{Direct image of parabolic vector bundles}

Let $f\,:\,X\,\longrightarrow\, Y$ be a finite flat map of smooth projective curves over $\C$. As before,
$D_X\,=\,(x_1,\,x_2,\,\cdots,\, x_m)$ with $x_i\,\in\, X$ and $x_i\,\neq\, x_k$ for all $i\,\neq\, k$. Take
${\bf r}\,=\,(r_1,\,r_2,\, \cdots,\,r_m)\,\in \,\N^m$. Denote by $R_f$ the ramification locus of $f$.
Let $E_*$ be a parabolic vector bundle on $(X,\,D_X)$ whose underlying vector bundle is $E$. 
In \cite[\S~4]{Alfaya_Biswas} the parabolic direct image $f_*E_*$ was defined on $(Y,\,f(D_X\cup R_f))$ with
underlying vector bundle $f_*E$. We briefly recall the construction: For $y\,\in\, Y$, let $f^{-1}(y)
\,=\,\bigsqcup\limits_j\Spec \frac{\mc O_{X,p_j}}{\varpi^{e_j}_{p_j}}$. 
For each $p_j$ we have the filtration coming from the parabolic structure of $E_*$ given by
$$E_{X,p_j}\,=\,E^0_j\,\supset\, E^1_j\,\supset\, \cdots \,\supset\, E^{r_j}_j\,=\,\varpi_{p_j}E_{X,p_j}\,.$$
This induces a filtration 
$$
E_{X,p_j}\, =\, E^0_j\,\supset\, E^1_j\,\supset\, \ldots \,\supset\, E^{r_j}_j
\,=\,\varpi_{p_j}E_{X,p_j}\, \supset\, \varpi_{p_j}E^1_j\,\supset\, \cdots \,\supset\,
 \varpi_{p_j}E^{r_j}_i
$$
\begin{equation}\label{eqn_direct_image_parabolic}
=\, \varpi^2_{p_j}E_{X,p_j}\,\supset\, \varpi^2_{p_j}E^1_j\,\supset\, \ldots \,\supset\,
\varpi^{e_j-1}_{p_j}E^{r_j}_i\,=\, \varpi^{e_j}_{p_j}E_{X,p_j}.
\end{equation}
Define $$E^q_j\ :=\ \varpi_{p_j}^lE^k_j \hspace{.5cm}\, \,\, {\rm if }\, \, q\,\in\,\left[ \frac{r_jl+k}{r_je_j},\, \frac{r_jl+k+1}{r_je_j}\right)\cap \Q$$
for $0\,\leq\, l\,\leq\, e_j-1$ and $0\,\leq\, k\,\leq\, r_j-1$.
We get the parabolic structure on $$(f_*E)_{Y,y}\ =\ \bigcap\limits_j E_{X,p_j}\ \subset\ E_{\eta_X}$$ 
(here $\eta_X$ is the generic point of $X$) with weight in $\frac{1}{{\rm lcm}(r_je_j)}$ by taking intersection
of the above filtrations: For $q\,\in\, \frac{1}{\rm lcm(r_je_j)}\Z \cap [0,\,1]$, if $q\,\in\,
[\frac{a-1}{r_je_j},\, \frac{a}{r_je_j})$ define 
$E^q_y$ to be $\bigcap\limits_j E^q_j$. 

\subsection{Pullback of parabolic vector bundles}\label{section_pullback}

Let $f\,:\,X\,\longrightarrow\, Y$ be a finite flat map of smooth curves over $\C$. Take $n$ distinct
ordered points $D_Y\,=\,(y_1,\,y_2,\,\cdots,\, y_n)$ of $Y$.
Also, take ${\bf s}\,=\,(s_1,\,s_2,\,\cdots,\, s_n)\,\in\, \N^{n}$. Consider a parabolic vector bundle $F_*$ on
$(Y,\, D_Y)$ with underlying vector bundle $F$. In \cite[\S~3]{Alfaya_Biswas} the pullback $f^*F_*$
on $(X,\,f^{-1}(D_Y)_{\rm red})$ was constructed; the construction is recalled.
If $F$ is a line bundle, then the underlying line bundle of $f^*F_*$ is given by 
$$F\otimes \mc O(\sum\limits_{y\in D_Y} \sum\limits_{x\in f^{-1}(y)} \lfloor \alpha_ye_x\rfloor x),$$
where $\alpha_y$ is the weight of $F_{Y,y}$ and $e_x$ is the ramification degree at $x$. The weight of $f^*F_*$
at $x$ is defined as $\{\alpha_ye_x\}$.

If $F$ is a vector bundle of rank at least two, choose an open covering
$U_1,\,\cdots,\,U_m$ of $Y$ such that 
$$F_*\big\vert_{U_j}\ =\ \bigoplus\limits_k L(j,k)_* ,$$
where $L(j,k)_*$ are parabolic line bundles. Define $f^*(F_*\big\vert_{U_j})\ :=\ \bigoplus\limits_k f^*L(j,k)_*$. Note
that $f^*(F_*\big\vert_{U_j})\big\vert_{f^{-1}(U_j)\setminus (f^{-1}(D_X)\cup R_f)}$ is canonically identified with
$f^*F\big\vert_{f^{-1}(U_j)\setminus (f^{-1}(D_Y)\cup R_f)}$ (with the trivial parabolic structure). Using these
identifications it can be checked that $f^*(F_*\big\vert_{U_j})$ glue together to give a parabolic vector
bundle on $X$.

\section{Direct image of vector bundles on root stacks}

Take $f\,:\, X\,\longrightarrow Y$,
$$D_X\,=\,(x_1,\,x_2,\, \cdots,\,x_m),\ \ \, D_Y\ =\ (y_1,\,y_2,\,\cdots,\, y_n)$$ and
$${\bf r}\ =\ (r_1,\,r_2,\, \cdots,\, r_m),\ \ \, {\bf s}\ =\ (s_1\, ,s_2,\, \cdots,\, s_n)$$
such that $f^{-1}(\sum s_jy_j)=R_f+\sum r_ix_i$. Let $$\X\ :=\ \sqrt[{\bf r}]{D_X/X}$$ be the root stack associated to the
pair $({\bf r},\, D_X)$ and let $$\Y\ :=\ \sqrt[{\bf s}]{D_Y/Y}$$ be the root stack associated to the
pair $({\bf s},\, D_Y)$ \cite[Definition 2.1]{Borne_Laaroussi}. Suppose we have a commutative diagram
\begin{equation}\label{eqn_diagram_root_stacks}
\begin{tikzcd}
\X \ar[r,"\widetilde f"] \ar[d,"\pi_X"] & \Y \ar[d,"\pi_Y"] \\
X \ar[r,"f"] & Y
\end{tikzcd}
\end{equation}
such that $\t f$ is \'etale. 

\begin{theorem}\label{main theorem} 
The two functors 
 $${\rm Vect}(\X)\ \xrightarrow{\,\,\,{\widetilde f}_*\,\,\, }\ {\rm Vect}(\Y)\ \xrightarrow{\,\,\,\sim\,\,\,}
\ {\rm Par}_{\frac{1}{\bf r}}(Y,\,D_Y),
$$
$$
{\rm Vect}(\X)\ \xrightarrow{\,\,\,\sim\,\,\,}\ {\rm Par}_{\frac{1}{\bf r}}(X,\,D_X)\ \xrightarrow{\,\,\,f_*\,\,\,}
\ {\rm Par}_{\frac{1}{\bf s}}(Y,\,D_Y)$$
are isomorphic. Similarly, the two functors 
$${\rm Vect}(\Y)\ \xrightarrow{\,\,\,{\widetilde f}^*\,\,\,}\ {\rm Vect}(\X)\ \xrightarrow{\,\,\,\sim\,\,\,}\
{\rm Par}_{\frac{1}{\bf r}}(X,\,D_X),
$$
$$ 
{\rm Vect}(\Y)\ \xrightarrow{\,\,\,\sim\,\,\,}\ {\rm Par}_{\frac{1}{\bf s}}(Y,\, D_Y)\ \xrightarrow{\,\,\,f^*
\,\,\,}\ {\rm Par}_{\frac{1}{\bf r}}(X,\, D_X)$$
are isomorphic. 
\end{theorem}

\begin{proof}
Let $\E$ be a vector bundle on $\X$ and $\F$ a vector bundle on $\Y$. We need to show that there are canonical
isomorphisms of parabolic vector bundles
$$\widehat{(\widetilde{f}_*\E)}\ \cong\ f_*(\widehat{\E}), \ \ \ \widehat{(\widetilde{f}^*\F)}\ \cong\
f^*(\widehat{\F})$$ on $(Y,\,D_Y)$ and $(X,\,D_X)$ respectively.
We begin by setting up some notation.

Take $y\,\in\, Y$, and let 
 $$f^{-1}(y)\ =\ \bigsqcup\limits_{j=1}^l \Spec \frac{\mc O_{X,p_j}}{\varpi^{e_i}_{p_j}}.$$ 
Let $A\,:=\, \mc O_{Y,y},\, f^{-1}(\Spec A)\,=\,\Spec B$ and $B_j\,:=\, \mc O_{X,p_j}$, so we have the diagram
\[
\begin{tikzcd}
\Spec B_i \ar[r] & \Spec B \ar[r] \ar[d] & \Spec A \ar[d] \\
& X \ar[r,"f"] & Y
\end{tikzcd}
\]
Denote by $\varpi$ a uniformizing parameter of $y$ and by $\varpi_j$ a uniformizing parameter of $p_j$.
If $p_j\,=\,x_i$ for some $i$, define $r_j\,:=\,r_i$; otherwise, define $r_j\,:=\,1$. If
$y\,=\,s_i$ for some $i$, define $s\,:=\,s_i$; otherwise, define $s\,=\,1$. 

Then \eqref{eqn_diagram_root_stacks} induces a commutative diagram
\begin{equation}\label{a1}
\begin{tikzcd}
\X_j\,:=\, \left[{\Spec}~ \dfrac{B_j[T_j]}{(T_j^{r_j}-\varpi_{j})} \middle/\mu_{r_j} \right] \ar[r,"\psi_2"] \ar[d] & \Y_y:=\left[{\Spec} \dfrac{A[T]}{(T^{s}-\varpi)}\middle/\mu_{s} \right] \ar[d] \\
\Spec B_i \ar[r] & \Spec A
\end{tikzcd}
\end{equation}
Now consider the composition of morphisms
\begin{equation}\label{eqn_morphism_scheme_to_root_stack}
{\Spec}~ \dfrac{B_j[T_j]}{(T_j^{r_j}-\varpi_{j})} \,\longrightarrow\, \X_j
\,\longrightarrow\, \Y_y .
\end{equation}
Recall that any morphism from a scheme $S$ to the root stack $\Y_y$
corresponds to data 
$$(S\xrightarrow{g} \Spec A,\, \mc L,\, P,\, \rho),$$ 
where $\mc L$ is a line bundle on $S$, $P$ a global section of $\mc L$ and $\rho\,:\,\mc L^{\otimes s}
\,\xrightarrow{\,\,\,\sim\,\,\,}\, \mc O_S$ an isomorphism such that $P^{\otimes s}\,\longmapsto
\, g^*\varpi$ \cite[\S 10.3.9]{Olsson} (see \eqref{a1}). In our case,
the data corresponding to the morphism (\ref{eqn_morphism_scheme_to_root_stack}) is given by 
$$(\Spec \frac{B_j[T_j]}{(T_j^{r_j}-\varpi_{j})}\to \Spec A,\ \frac{B_j[T_j]}{(T_j^{r_j}-\varpi_{j})},\
P,\ v).$$
Here $P\,\in\, \frac{B_j[T_j]}{(T_j^{r_j}-\varpi_{j})}$ with $P^{s}\,=\,v\varpi$, where $v$ is a unit in
$\frac{B_j[T_j]}{(T_j^{r_j}-\varpi_{j})}$.
Denote $\varpi_j^{e_j}u\,=\,\varpi$, where $u$ is a unit in $B_j$. Let $P\,=\,T_j^av'$, where $v'$ is a unit in
$\frac{B_j[T_j]}{(T_j^{r_j}-\varpi_{j})}$. Then we have
\begin{equation}\label{a2}
P^s\ =\ T_j^{as}v'^s\ =\ v\varpi\ =\ v\varpi_j^{e_j}u\ =\ vT_j^{r_je_j}u .
\end{equation}
In particular, this implies that $as\,=\,r_je_j$ and $v'^s\,=\,vu$.

\begin{claim}\label{claim}
We have a $2$-fiber diagram
\begin{equation}\label{eqn_claim}
\begin{tikzcd} 
{\Spec}~ \dfrac{B_j[T_j,X]}{(T_j^{r_j}-\varpi_{j}, X^s-u)} \ar[r,"\phi_1"] \ar[d,"\pi_1"] & \left[ {\Spec}~ \dfrac{B_j[T_j,X]}{(T_j^{r_j}-\varpi_{j}, X^s-u)} \middle/ \mu_{r_j} \right] \ar[r,"\psi_1"] \ar[d,"\pi_2"] & {\Spec} \dfrac{A[T]}{(T^{s}-\varpi)} \ar[d,"\pi_3"] \\
 {\Spec}~ \dfrac{B_j[T_j]}{(T_j^{r_j}-\varpi_{j})} \ar[r,"\phi_2"] & \left[{\Spec}~ \dfrac{B_j[T_j]}{(T_j^{r_j}-\varpi_{j})} \middle/\mu_{r_j} \right] \ar[r,"\psi_2"] & \left[{\Spec} \dfrac{A[T]}{(T^{s}-\varpi)}\middle/\mu_{s} \right]
\end{tikzcd}
\end{equation}
Here the $\mu_{r_j}$-action on ${\Spec}~ \dfrac{B_j[T_j,X]}{(T_j^{r_j}-\varpi_{j}, X^s-u)}$ is defined by 
$$\exp\left({\frac{2\pi\sqrt{-1}}{r_j}}\right)\cdot T_j\,:=\,\exp\left({\frac{2\pi\sqrt{-1}}{r_j}}\right) T_j,\ \ \,
\exp\left({\frac{2\pi\sqrt{-1}}{r_j}}\right)\cdot X
\,:=\,\exp\left({-\frac{2\pi\sqrt{-1}a}{r_j}}\right)X,$$
and the morphism $\psi_1\circ \phi_1$ is defined by $T\, \longmapsto\, T^a_jX$ (see \eqref{a2}).
\end{claim}

\begin{proof}[Proof of claim]
Note that the morphism 
\begin{equation}\label{eqn_etale_map}
{\Spec}~ \dfrac{B_j[T_j,X]}{(T_j^{r_j}-\varpi_{j}, X^s-u)}\ \longrightarrow
\ {\Spec} \dfrac{A[T]}{(T^{s}-\varpi)}
\end{equation}
is well defined because 
$(T^a_jX)^s\,=\,T^{as}X^s\,=\,T^{r_je_j}u\,=\,\varpi_j^{e_j}u\,=\,\varpi$, and by the definition of
$\mu_{r_j}$-action on ${\Spec}~ \frac{B_j[T_j,X]}{(T^{r_j}-\varpi_{j}, X^s-u)}$, the element $T^a_jX$
is $\mu_{r_j}$-invariant. Hence (\ref{eqn_etale_map}) is $\mu_{r_j}$-invariant.
The $2$-commutativity of the diagram
\[
\begin{tikzcd} 
{\Spec}~ \dfrac{B_j[T_j,X]}{(T^{r_j}-\varpi_{j}, X^s-u)} \ar[r] \ar[d] & {\Spec} \dfrac{A[T]}{(T^{s}-\varpi)} \ar[d] \\
 {\Spec}~ \dfrac{B_j[T_j]}{(T_j^{r_j}-\varpi_{j})} \ar[r] & \left[{\Spec} \dfrac{A[T]}{(T^{s}-\varpi)}\middle/\mu_{s} \right]
\end{tikzcd}
\]
follows from the fact that the two compositions correspond to the following two:
$$(\frac{B_j[T_j,X]}{(T_j^{r_j}-\varpi_{j}, X^s-u)},\, P\,=\,T_j^av',\, v),\ \ \,
(\frac{B_j[T_j,X]}{(T_j^{r_j}-\varpi_{j}, X^s-u)},\, T^aX,\, 1),$$ 
and these two are actually isomorphic in the sense of \cite[\S~10.3.9]{Olsson}, with the isomorphism given by
multiplication with $v'^{-1}X$. Since \eqref{eqn_etale_map} is $\mu_{r_j}$-invariant and $\pi_1$ is $\mu_{r_j}$-equivariant, we have a commutative diagram as in (\ref{eqn_claim}). The diagram is fibered because both $\pi_1$ and $\pi_2$ are $\mu_{r_j}$-torsors. This completes the proof of the claim. 
\end{proof}

Since $\psi_1$ is \'etale, it follows that the map $\psi_1\circ \phi_1$ is also \'etale. This implies that
$a\,=\,1$ (see \eqref{a2}). Let us denote by $\E_j$
the pullback of the sheaf $\E$ on $\X$ to $\X_j$ and the decomposition of $\mu_{r_j}$-equivariant sheaf
$\phi_2^*\E_j$ by $$E_{j,0}\,\oplus\, E_{j,1}\,\oplus\,\cdots\,\oplus\, E_{j,r_j-1}$$ 
(see \S~\ref{section_Parabolic bundles and bundles on root stacks}). We want to understand the decomposition
of the $\mu_s$-equivariant sheaf $\pi_{3}^*\psi_{2*}\E_j$. By flat base change, 
$$\pi_{3}^*\psi_{2*}\E_j\ \cong\ \psi_{1*}\pi_2^*\E_j.$$ Now note that we have a factorization of $\psi_1$ as
\[
{\Spec}~ \dfrac{B_j[T_j,X]}{(T_j^{r_j}-\varpi_{j}, X^s-u)}\ \xrightarrow{\,\,\,\phi_1} \,\,\,
\Bigg[ {\Spec}~ \dfrac{B_j[T_j,X]}{(T_j^{r_j}-\varpi_{j}, X^s-u)} \Bigg/ \mu_{r_j} \Bigg]
\]
\[
\xrightarrow{\,\,\,\psi_1'\,\,\,} {\Spec}~ \dfrac{B_j[T_j,X]}{(T_j^{r_j}-\varpi_{j}, X^s-u)}
\Bigg/ \mu_{r_j} \xrightarrow{\psi_1''} {\Spec} \dfrac{A[T]}{(T^{s}-\varpi)}.
\]
So we have $\psi_{1*}\pi_2^*\E_j\,=\,\psi_{1*}''\circ \psi_{1*}'\pi_2^*\E_j$. Now recall that
$\psi_{1*}'\pi_2^*\mc E_j$ is isomorphic to the invariant direct image $$((\psi_1'\circ \phi_1)_*\phi^*_1
\pi_2^*\E_j)^{\mu_{r_j}}.$$ Hence it follows that 
$\psi_{1*}\pi_2^*\E_j$ is the $\mu_{r_j}$-invariant direct image of $(\pi_2\circ \phi_1)^*\mc E_j
\,=\,(\phi_2\circ \pi_1)^*\E_j$. Now we have that 
$$(\phi_2\circ \pi_1)^*\E_j\ =\ \sum_{k=0}^{s-1}(E_{j,0}\oplus E_{j,1}\oplus\cdots\oplus E_{j,r_j-1})X^k.$$
Therefore, the $\mu_{r_i}$-invariant direct image is given by
\begin{equation}\label{eqn_local}
\psi_{1*}\pi^*_2\E_j\ =\ \sum_{l=0}^{e_j-1}\sum_{k=0}^{r_j-1}(E_{j,k})X^{r_jl+k}.
\end{equation}
Since $T\,\longmapsto\, T_jX$, multiplication by $T$ gives the filtration
\[
\begin{tikzcd}
E_{j,0}\ar[r,hook, "\times T_j" ] & E_{j,1}\ar[r,hook, "\times T_j"] & \ldots & E_{j,0}\ar[r,hook, "\times T_j"] & E_{j,1}\ar[r,hook, "\times T_j"] & \ldots & \ar[r,hook, "\times T_j"] & E_{j,0}\,.
\end{tikzcd}
\]
Note that this is same as the parabolic structure on $(f_*\widehat{\mathcal E})_{Y,y}$ as defined in
\eqref{eqn_direct_image_parabolic}. Define $\X_y\ :=\ \pi_X^{-1}\Spec B$.
Then the following diagram is fibered:
\[
\begin{tikzcd}
\X_y \ar[r,"f_y"] \ar[d] & \Y_y \ar[d] \\
\X \ar[r] & \Y 
\end{tikzcd}
\]
We also have a commutative diagram
\[
\begin{tikzcd}
\bigsqcup\limits_{j,j'} \Spec K(B) \ar[r] \ar[d,equal] & \bigsqcup\limits_j \X_j \ar[r,"i_j"] \ar[d] & \X_{y} \ar[r,"f_y"] \ar[d] & \Y_y \ar[d] \\
\bigsqcup\limits_{j,j'} \Spec K(B) \ar[r] & \bigsqcup\limits_j \Spec B_j \ar[r] & \Spec B \ar[r] & \Spec A
 \end{tikzcd}
\]
Here the middle and left most squares are fibered. In particular, $i_j$'s are \'etale. This implies that we
have an equalizer
\begin{equation}\label{eqn_global}
0 \, \longrightarrow\, f_{y*}\mc E_y\, \longrightarrow\, \bigoplus(f_y\circ i_j)_*\E_j\,
\rightrightarrows\, \bigoplus \E_{\eta}.
\end{equation}
Now consider the $\mu_s$-equivariant module $\sum_{l=0}^{e_j-1}\sum_{k=0}^{r_j-1}(\bigcap\limits_j 
E_{j,k})X^{r_jl+k}$ on $ {\Spec} \frac{A[T]}{(T^{s}-\varpi)}$. By (\ref{eqn_local}) its associated vector 
bundle on $\Y_y$ fits into the same exact sequence (\ref{eqn_global}). Hence we get that the associated 
parabolic structure of $\widehat{\widetilde f_*\E}$ at $y$ is given by $\frac{r_jl+k}{s}\, \longmapsto\,
(\bigcap\limits_j E_{j,k})$. This completes the proof of the first part of the theorem.

For the second part, let $U_1,\,\cdots,\,U_m$ be an open cover of $Y$ such that $$\widehat{\F}\big\vert_{U_j}
\ =\ \bigoplus\limits_k L(j,k)_* ,$$ 
where $L(j,k)_*$ are parabolic line bundles on $U_j$. Hence $$\F\big\vert_{\pi^{-1}(U_j)}
\ =\ \bigoplus\limits_k \mc L(j,k),$$
where $\mc L(j,k)$ is the line bundle on $\pi^{-1}(U_j)$ associated to the parabolic line bundle
$L(j,k)$. This implies that
$$\widetilde{f}^*\F\big\vert_{(\pi_Y\circ \widetilde{f})^{-1}(U_j)}\ =\ \bigoplus\limits_k \widetilde{f}^*\mc L(j,k)$$
and therefore, it suffices to establish the statement for case when ${\rm rank}~F\,=\,1$.

Let $x\,=\, x_i\,\in\, D_X$, $f(x)\,=\,y$ and $r\,=\,r_i$. If $y\,=\,y_j$ for some $j$, define $s\,=\,s_j$;
otherwise define $s_j\,=\, 1$. Let $\varpi_x$ and $\varpi_y$ be uniformizing parameters of $x$ and $y$
respectively. Let $e$ be the ramification index of $f$ at $x$. Let $\varpi_x^eu\,=\,\varpi_y$. Let $A
\,=:\,\mc O_{Y,y}$ and $B\,=:\,\mc O_{X,x}$. Then by Claim \ref{claim} we have that $s\,=\,re$ and a $2$-fiber
diagram
 \begin{equation}\label{eqn_claim_1}
\begin{tikzcd} 
{\Spec}~ \dfrac{B[T,X]}{(T^{r}-\varpi_{x}, X^s-u)} \ar[r,"\phi_1"] \ar[d,"\pi_1"] & \left[ {\Spec}~ \dfrac{B[T,X]}{(T^{r}-\varpi_{x}, X^s-u)} \middle/ \mu_{r} \right] \ar[r,"\psi_1"] \ar[d,"\pi_2"] & {\Spec} \dfrac{A[T]}{(T^{s}-\varpi_y)} \ar[d,"\pi_3"] \\
 {\Spec}~ \dfrac{B[T]}{(T^{r}-\varpi_{x})} \ar[r,"\phi_2"] & \left[{\Spec}~ \dfrac{B[T]}{(T^{r}-\varpi_{x})} \middle/\mu_{r} \right] \ar[r,"\psi_2"] & \left[{\Spec} \dfrac{A[T]}{(T^{s}-\varpi)}\middle/\mu_{s} \right]
\end{tikzcd}
\end{equation}
where $\psi_1\circ \phi_1$ is given by $T\, \longmapsto\, TX$. 
Denote the pullback of $\F$ to $\Y_y\,:=\, \Spec A \times \Y$ by $\F_y$.
Now if $\widehat{\F}_y$ has parabolic weight $\alpha$ with underlying bundle $F_y$ then $\pi_3^*\mc F_y$
is the module $\bigoplus\limits_{i=0}^{s-1} F^i_y$ with $F^i_y\,=\,F_y$ for $i\,\leq\, s\alpha$ and
$F^i_y\,=\,\varpi_yF_y$ for $i\,>\,s\alpha$.
Note that 
$$\bigoplus\limits_{i=0}^{s-1} F^i_y\ =\ F_y\otimes (T^{-s\alpha})$$
as $\frac{A[T]}{(T^s-\varpi_y)}$ modules. Hence $(\pi_3\circ \psi_1\circ \phi_1)^*\F_y$ is the module 
$$(F_y\otimes_A B)\otimes ((TX)^{-s\alpha})\ =\ (F_y\otimes_A B)\otimes (T^{-s\alpha}).$$
Since $\pi_1$ is a $\mu_s$-torsor, the module $(\psi_2\circ\phi_2)^*\F_y$ is the $\mu_s$-invariant direct image of $(\psi_1\circ \phi_1\circ \pi_3)^*\F_y$
. Therefore we get that
\begin{align*}
(\psi_2\circ\phi_2)^*\F_y\,=\, & (F_y\otimes_A B)\otimes (T^{-s\alpha}) \\
=\, & (F_y\otimes_A B)\otimes (T^{-r\lfloor\frac{s\alpha}{r}\rfloor}.T^{-r\{\frac{s\alpha}{r}\}}) \\
=\, & (F_y\otimes_A B)\otimes (T^{-r\lfloor e\alpha\rfloor}.T^{-r\{e\alpha\}})\,. \\
=\, & (F_y\otimes_A B)\otimes (\varpi_x^{-\lfloor e\alpha\rfloor}.T^{-r\{e\alpha\}}). \\
\end{align*}
Therefore, the underlying module of $\psi_2^*\F_y$ is $F_y\otimes_A \varpi_x^{-\lfloor e\alpha\rfloor}B$ and the
parabolic weight is $\{e\alpha\}$, which is same as the underlying bundle and parabolic weight of the pullback
parabolic line bundle $f^*\widehat{\F}_y$, as discussed in \S \ref{section_pullback}. This completes the proof of the theorem.
 \end{proof}
Let $L_*$ be a parabolic line bundle of parabolic degree zero on $Y$ and let $\mc L$ be the corresponding
line bundle on $\Y$. An $L_*$-valued parabolic symplectic (respectively, orthogonal) bundle is a parabolic
vector bundle
$F_*$ on $Y$ together with a map $\widehat{\phi}\,:\,F_*\otimes F_*\,\longrightarrow\, L_*$ which is
antisymmetric (respectively, symmetric) such that the induced map $F_*\,\longrightarrow\, F_*^{\vee}\otimes L_*$
is an isomorphism \cite[Definition 2.1]{Alfaya_Biswas_Machu}. 
By \cite[Proposition 4.1]{Alfaya_Biswas_Machu}, the parabolic vector bundle $f^*F_*$ endowed with the map
$f^*\widehat \phi$ is $f^*L_*$-valued parabolic symplectic (respectively, orthogonal) bundle. Let $\F$ be
the bundle on $\Y$ corresponding to $F$ and let $\phi\,:\,\F\otimes \F\,\longrightarrow\, \mc L$ be the morphism
corresponding to $\widehat{\phi}$. By \cite[Theorem 4.0.11]{Chakraborty_Majumdar}, $\F$ endowed with $\phi$
is a symplectic (respectively, orthogonal) bundle on $\Y$. By Theorem \ref{main theorem}, we have the following:

\begin{corollary}\label{cor1}
The parabolic symplectic (respectively, orthogonal) bundle $(f^*F_*,\,f^*\widehat\phi)$ is isomorphic to
the parabolic symplectic (respectively, orthogonal) bundle $(\widehat{\widetilde f^*\F},\,
\widehat{\widetilde f^*\phi})$.
\end{corollary}

Similarly, we can define direct image of parabolic symplectic (respectively, orthogonal) bundles as follows
\cite[\S~
5.1]{Alfaya_Biswas_Machu}. Let $\E$ is a vector bundle on $\X$ with corresponding parabolic vector bundle $E_*$ and 
let $\widehat{\psi}\,:\,E_*\otimes E_*\,\longrightarrow\, f^*L_*$ is parabolic symplectic (respectively,
orthogonal) structure. Let 
$\widehat{\psi'}\,:\,E_*\,\xrightarrow{\,\,\,\sim\,\,\,}\, E^{\vee}_*\otimes f^*L$ be the induced map. By
projection formula, we 
get that a map $f_*\widehat{\psi'}\,:\,f_*E_*\,\xrightarrow{\,\,\,\sim\,\,\,}\, (f_*E_*)^{\vee}\otimes L$
which induces a map $f_*E_*\otimes f_*E_*\,\longrightarrow\, L_*$. By \cite[Lemma 5.1]{Alfaya_Biswas_Machu}
$(f_*E_*,f_*\widehat{\psi'})$ is a 
parabolic symplectic (respectively, orthogonal) bundle on $(Y,D_Y)$. Let us denote $\psi'\,:\,
\E\,\xrightarrow{\,\,\,\sim\,\,\,}\, 
\E^{\vee}\otimes \widetilde f^*L$ be morphism on $\X$ corresponding to the morphism $\widehat{\psi'}$. $\psi'$ makes $\mathcal E$ a symplectic
(respectively, orthogonal) bundle on $\mathcal{X}$ by
\cite[Theorem 4.0.11]{Chakraborty_Majumdar}.Then we 
have the following:

\begin{corollary}\label{cor2}
The parabolic symplectic (respectively, orthogonal) bundle $(f_*E_*,\,f_*\widehat \psi')$ is
isomorphic to the parabolic symplectic (respectively, orthogonal) bundle $\widehat{\widetilde f_*\E},\,
\widehat{\widetilde f_*\psi'}$.
\end{corollary}

\section*{Acknowledgements}

We thank the referee for useful comments. The first-named author is
partially supported by a J. C. Bose Fellowship (JBR/2023/000003).

\end{document}